\documentclass[12pt]{amsart}
\usepackage{amsmath,amssymb}

\theoremstyle{plain}
\newtheorem{thm}{Theorem}[section]
\newtheorem{theorem}[thm]{Theorem}

\newtheorem{lemma}[thm]{Lemma}
\newtheorem{corollary}[thm]{Corollary}
\newtheorem{proposition}[thm]{Proposition}
\theoremstyle{definition}

\newtheorem{remark}[thm]{Remark}

\newtheorem{definition}[thm]{Definition}

\numberwithin{equation}{section}
\newcommand{\0}{{\mathcal O}}

\newcommand{\sO}{{\mathcal O}}

\newcommand{\C}{{\mathbb C}}

\newcommand{\BP}{{\mathbb P}}
\newcommand{\pit}{{\mathbb P}}
\newcommand{\Q}{{\mathbb Q}}

\newcommand{\BS}{{\mathbb S}}

\def\cit{{\mathbb C}}

\newcommand{\fg}{{\mathfrak g}}
\newcommand{\fsl}{{\mathfrak s}{\mathfrak l}}
\newcommand{\fspin}{\mathfrak{spin}}

\newcommand{\aut}{{\mathfrak a}{\mathfrak u}{\mathfrak t}}

\newcommand\sd{\!>\!\!\!\! \lhd \:}

\def\Gr{\mathop{\rm Gr}\nolimits}

\def\Lag{\mathop{\rm Lag}\nolimits}
\def\Sym{\mathop{\rm Sym}\nolimits}

\def\Hom{\mathop{\rm Hom}\nolimits}

\def\PP{\mathbf P}
\def\CC{\mathbb{C}}\def\OO{\mathbb{O}}

\title[On Fano complete intersections in rational homogeneous varieties]{On Fano complete intersections in rational homogeneous varieties}
\author{Chenyu Bai, Baohua Fu and Laurent Manivel}
\begin{document}
\maketitle

\begin{abstract}
Complete intersections inside rational homogeneous varieties provide interesting examples of Fano manifolds. For example, if $X = \cap_{i=1}^r D_i \subset G/P$ is a  general complete intersection of $r$ ample divisors such that $K_{G/P}^* \otimes \0_{G/P}(-\sum_i D_i)$ is ample, then $X$ is Fano.
We first classify these Fano complete intersections which are locally rigid.  It turns out that most of them are hyperplane sections.  We then classify general hyperplane sections which are quasi-homogeneous.
\end{abstract}

\section{Introduction}
We work within the category of complex projective varieties, unless stated otherwise.
Rational homogeneous varieties are among the simplest algebraic varieties, and a better understanding of them is always a motivation for the development of algebraic geometry. For example, the solution by Mori of the Hartshorne conjecture characterizes projective spaces by the ampleness of its tangent bundle, which  is a milestone of the minimal model program.  A more recent conjecture of Campana-Peternell claims that rational homogeneous varieties are the only smooth rational varieties with nef tangent bundle, which is still far from resolved.

Complete intersections in rational homogeneous varieties provide many interesting examples of Fano varieties. It is expected by Hartshorne that all
smooth subvarieties in $\BP^n$ of small codimension are complete intersections, which is again far from resolved.
In this paper, we will study two geometrical properties of Fano complete intersections in rational homogeneous varieties: local rigidity and quasi-homogeneity.

Recall that a smooth projective variety $X$ is said {\em locally rigid} if for any smooth deformation $\mathcal{X} \to B$ with $\mathcal{X}_0 \simeq X$,  we have $\mathcal{X}_t \simeq X$ for $t$ in a small (analytic) neighborhood of $0$.  By Kodaira-Spencer deformation theory,  if $H^2(X, T_X)=0$, then $X$ is locally rigid if and only if $H^1(X, T_X)=0$. For rational homogeneous varieties $G/P$, it is shown in \cite{B} (Theorem VII) that $H^i(G/P, T_{G/P})=0$ for all $i \geq 1$, hence they are locally rigid.
In \cite{BB}, the local rigidity is proven for Fano regular $G$-varieties.  The case of two-orbits varieties of Picard number one is studied in \cite{PP}.

Let $G/P$ be a rational homogeneous variety with $G$ simple and $X = \cap_{i=1}^r D_i \subset G/P$ a  smooth irreducible complete intersection of $r$ ample divisors. We assume that $K_{G/P}^* \otimes \0_{G/P}(-\sum_i D_i)$ is ample,  which implies that $X$ is  Fano.  When $G/P$ is of Picard number one, the converse holds, but in general this condition is stronger than the Fanoness of $X$ (cf. Remark \ref{r.Fano}).
  The main purpose of this paper is to classify such $X$ which are locally rigid.
By Kodaira-Nakano vanishing theorem, we have $H^q(X, T_X)=0$ for all $q \geq 2$. In particular, $X$ is locally rigid if and only if $H^1(X, T_X)=0$.
The main theorem of this paper is the following, which  generalizes Proposition 8.4 in \cite{FH3}, where a similar result is obtained in the case of hyperplane sections of irreducible Hermitian symmetric spaces.

\begin{theorem} \label{t.main}
Let $G/P$ be a rational homogeneous variety  with $G$ simple and $X = \cap_{i=1}^r D_i \subset G/P$ a smooth complete intersection of ample divisors.
Assume that $K_{G/P}^* \otimes \0_{G/P}(-\sum_i D_i)$ is ample.  Then $X$ is locally rigid if and only if $X$ is isomorphic to one of the following:
\begin{itemize}
\item[(i)]  $\BP^n$ or $\Q^n$;

\item[(ii)] a general hyperplane section of the following:
$$
 {\rm Gr}(2, n), {\rm Gr}(3, 6), {\rm Gr}(3,7),  {\rm Gr}(3,8),  $$  $$
\BS_5, \BS_6, \BS_7, {\rm Gr}_\omega(2, 6), {\rm Lag}(3, 6), F_4/P_4,
E_6/P_1, E_7/P_7;
$$

\item[(iii)] a general hypersurface of bidegree $(1,1)$ of $\BP(T_{\BP^2})$;

\item[(iv)]  a general codimension 2 linear section of ${\rm Gr}(2, 2k+1)$, $k\ge 2$;

\item[(v)] a general codimension 2 or  3 linear section of $\BS_5$;

\item[(vi)]  a general codimension 3 or 4 linear section of ${\rm Gr}(2,5)$.
\end{itemize}
\end{theorem}

Here $\Q^n$ denotes the
$n$-dimensional hyperquadric.  $\Gr(a, a+b)$ is the Grassmannian of $a$-dimensional
subspaces in an $(a+b)$-dimensional vector space. $\mathbb{S}_{n}$
is the spinor variety, parameterizing
$n$-dimensional isotropic linear subspaces in an orthogonal vector
space of dimension $2n$. ${\rm Gr}_\omega(2, 6)$ is the symplectic Grassmanian and $\Lag(3,6)$ is the Lagrangian Grassmannian, which parameterize, respectively, isotropic planes and
Lagrangian subspaces in a $6$-dimensional symplectic vector space.
For a simple Lie group $G$, we denote by $P_i$ the maximal parabolic subgroup of $G$ corresponding to the $i$-th root, where we use Bourbaki's numeration of simple roots.

This apparently disparate list can be explained in terms of Vinberg's theory of parabolic
prehomogeneous spaces \cite{pr}. Briefly, suppose that a node $n$ is chosen on a connected
Dynkin diagram $D$, such
that the complement of the node is the disjoint union of a Dynkin diagram of type $A_{k-1}$
(including $k=1$, $A_0$ being by convention the empty diagram) and a connected Dynkin diagram
$D_0$. The latter comes equipped
with a special node $n_0$, the node which was connected to $n$ in $D$.
The pair $(D_0,n_0)$ encodes a simply connected simple Lie group $G$ and a maximal
parabolic subgroup $P$, hence a homogeneous space $G/P$ embedded in $\PP V_P^*$, the
projectivization of a (dualized) fundamental representation. The fundamental fact then is
that $G\times GL_k$ acts on $V_P \otimes \CC^k$ with finitely many orbits. In particular
$G$ acts on $Gr(k,V_P)$ with only finitely many orbits, and therefore there exists only
a finite number of isomorphism types of codimension $k$ linear sections of $G/P$. In this
situation, the local rigidity of the general section can be expected, and this is exactly
what happens.

We illustrate below the cases that originate from $D=E_8$. To each admissible node we attached
the corresponding homogeneous space, with a superscript indicating the number $k$,
which is the codimension of the relevant linear sections.

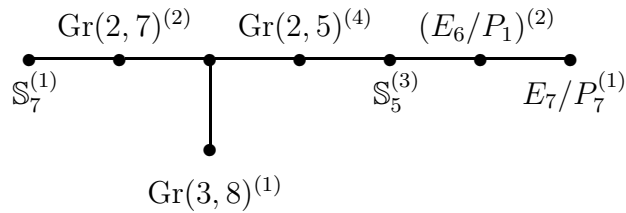
\setlength{\unitlength}{6mm}
\thicklines
\begin{picture}(15,5.9)(-6,-1.5)
\multiput(0,2)(2,0){7}{$\bullet$}\put(4,0){$\bullet$}
\multiput(0.3,2.2)(2,0){6}{\line(1,0){1.8}}
\put(4.2,0.3){\line(0,1){1.8}}
\put(-.2,1.2){$\BS_7^{(1)}$}
\put(.8,2.7){${\rm Gr}(2,7)^{(2)}$}
\put(4.8,2.7){${\rm Gr}(2,5)^{(4)}$}
\put(2.8,-1){${\rm Gr}(3,8)^{(1)}$}
\put(7.8,1.2){$\BS_5^{(3)}$}
\put(8.8,2.7){$(E_6/P_1)^{(2)}$}
\put(11.1,1.2){$E_7/P_7^{(1)}$}
\end{picture}

Taking all the connected
diagrams we get exactly the list of Theorem \ref{t.main}, except the codimension two
linear sections of $Gr(2,2k+1)$ (which for $k\ge 4$ would originate from the non
Dynkin diagrams $E_{2k+2}$). The general hypersurface of bidegree $(1,1)$ of $\BP(T_{\BP^2})$ is a complete intersection of two divisors of bidegree $(1,1)$ in $\BP^2 \times \BP^2$, which can be regarded as the linear section associated to the triple node in $E_6$.

Severi varieties are extremal projective varieties with remarkable projective geometrical properties, which are classified by Zak as follows:  the Veronese surface, minimal embeddings of $\BP^2 \times \BP^2, {\rm Gr}(2,6)$ and $E_6/P_1$.
 It is interesting to notice that a general hyperplane section of them is homogeneous, while their  general codimension 2 linear sections are locally rigid.

One remarks that in order to prove Theorem \ref{t.main}, we may assume that $X$ is a general complete intersection, since special ones have deformations to the general ones, hence they are not locally rigid.  On the other hand, special complete intersections may have much richer geometry which remains to be explored systematically. 
 One example is the 10-dimensional spinor variety $\BS_5$. Up to projective isomorphism there are
only two classes of smooth codimension 2 linear sections of $\BS_5$. It is shown in \cite{FH2}
(Remark 2.13) that the special ones contain a $\BP^4$ and are equivariant compactifications of $\C^8$,
which is not the case of the general ones. Of course the special codimension 2 section of
$\BS_5$ is not locally rigid, while the general one is locally rigid.  Surprisingly, we discover that the general codimension 2 linear section of $\BS_5$ is
one of the two-orbits varieties in \cite{Pa}, which is quasi-homogeneous (Proposition \ref{p.2S5}). In particular, we obtain two non-isomorphic quasi-homogeneous varieties (special and general linear sections of codimension 2 of $\BS_5$) which have the same VMRT at general points. This makes even more delicate the problem of recognition of  Fano varieties of Picard number one from its VMRT.

\smallskip
By \cite{A}, a general hyperplane section of $G/P$ (with $G$ simple) is homogeneous if and only if $G/P$ is isomorphic to $\BP^n, \Q^n, {\rm Gr}(2, 2k)$ or $E_6/P_1$. A natural question is: when is a general hyperplane section $X$ of $G/P$ quasi-homogeneous, i.e. ${\rm Aut}(X)$ acts on $X$ with an open orbit?  In this paper, we obtain the following classification.

\begin{theorem} \label{t.main1}
Let $G/P$ be a rational homogeneous variety of Picard number one and $X \subset G/P$ a general hyperplane section.  Then $X$ is quasi-homogeneous if and only if $G/P$ is isomorphic to one of the following
$$
\mathbb{P}^n,  \mathbb{Q}^n, {\rm Gr}(2, n), {\rm Gr}(3, 6), {\rm Gr}(3,7), $$
$$ \BS_5, \BS_6, \BS_7,  {\rm Gr}_\omega(2, 6), {\rm Lag}(3, 6), F_4/P_4, E_6/P_1, E_7/P_7.$$

\end{theorem}

An observation is that a general hyperplane section of $G/P$ is quasi-homogeneous if and only if it is locally rigid but not a hyperplane section of ${\rm Gr}(3,8)$. In general, there is no direct relation between the two properties.

Once the local rigidity is settled, the next question is whether the varieties in Theorem \ref{t.main} are rigid?
Namely if we have a smooth K\"ahler deformation $\mathcal{X} \to B$ such that $\mathcal{X}_t \simeq X$ for all $t \neq 0$, does this imply that $\mathcal{X}_0\simeq X$? This problem is already difficult for $G/P$ and was solved by Hwang and Mok (cf.  \cite{HM}). It seems very interesting to extend their results to the varieties in Theorem \ref{t.main}.  Note that by the previous discussions, a general codimension 2  linear section of $\BS_5$ is locally rigid, but not rigid, as it has deformations to the special section.

\begin{remark}
 As is well-known, a smooth Fano complete intersection in $\BP^n$ is locally rigid if and only if it is isomorphic to $\BP^m$ or $\Q^m$ (cf. Proposition \ref{p.Pn}). If a homogeneous variety $G/P$ is a complete intersection in $G'/P'$, then we  only need to consider complete intersections in $G'/P'$.
 This is the reason why we introduce the following convention:
we say that $G/P$ satisfies $\clubsuit$ if  $G/P$  is not isomorphic to one of the following: $\BP^n, \Q^n,  C_\ell/P_2 (\ell \geq 3), F_4/P_4, \BP(T_{\BP^m})$.
\end{remark}

{\em Acknowledgements:} We are grateful to Michel Brion for Lemma \ref{l.reductive} and for the reference \cite{R}.  Baohua Fu is supported by National Natural Science Foundation of China (No. 11621061, 11688101 and 11771425). Part of this work is done during the visit of Baohua Fu to the CMSA of Harvard University and he would like to thank the CMSA for the hospitality.

\section{Reduction to Picard number one case}

Let $G$ be a semi-simple Lie group of rank $\ell$ with Lie algebra $\fg$. We fix a Borel subgroup and a maximal torus. Let $\{\alpha_1, \cdots, \alpha_\ell\}$ be the set of simple roots.  The fundamental weights are denoted by
$\{\lambda_1, \cdots, \lambda_\ell\}$.  Every standard parabolic subgroup $P$ in $G$ is determined by a subset of indexes  $\Delta \subset \{1, \cdots, \ell\}$, with the property that $\alpha_i \notin {\rm Lie}(P)$ for all $i \in \Delta$. We have a natural identification
$$
{\rm Pic}(G/P)  = \{ \sum_{i \in \Delta} n_i \lambda_i | n_i \in \mathbb{Z}\}.
$$

For $\lambda = \sum_{i \in \Delta} n_i \lambda_i $, we denote by $L^\lambda$ the corresponding line bundle. It is well-known that $L^\lambda$ is ample if and only if $n_i >0$ for all $i\in\Delta$
(and in this case, it is very ample). In particular, there exists a minimal ample line bundle $L_0$, which corresponds to  $\sum_{i \in \Delta}  \lambda_i$.  As a consequence, we have a minimal $G$-equivariant embedding $G/P \subset \BP(V_P^*)$, where $V_P = H^0(G/P, L_0)$.

By Kodaira vanishing theorem, we have
\begin{lemma} \label{l.O}
Let $G/P$ be a rational homogeneous variety and $L \in {\rm Pic}(G/P)$ an ample Line bundle. Assume that $K_{G/P}^*\otimes L^*$ is ample.
 Then $H^q(G/P, L^*\otimes A)=0$ for all $q >0$ and nef line bundle $A$.
\end{lemma}

We recall the following theorem from \cite{MS} (Theorem B), which plays a key role in our computations. Note that claim (0) holds for any smooth projective variety by the result of Wahl \cite{W}.
\begin{theorem} \label{t.MS}
Let $G/P$ be a rational homogeneous variety and $L \in {\rm Pic}(G/P)$  an ample line bundle. Then
\begin{itemize}
\item[(0)]  $H^0(G/P, T_{G/P} \otimes L^*)=0$  except for $(G/P, L) = (\BP^1, \0(2))$ or $(\BP^n, \0(1))$.
\item[(1)]  $H^1(G/P, T_{G/P} \otimes L^*)=0$ except the following cases
\begin{itemize}
\item[(a)] $H^1(\BP^1, T(-k)) \simeq  {\rm Sym}^{k-4} \C^2, k \geq 4 $;
\item[(b)] $H^1(\BP^2, T(-3)) \simeq  \C$;
\item[(c)] $H^1(\Q^n, T(-2)) \simeq  \C, n \geq 3$;
\item[(d)] $H^1(C_\ell/P_2, T_{C_\ell/P_2}(-1)) \simeq \C$;
\item[(e)] $H^1(F_4/P_4, T_{F_4/P_4}(-1)) \simeq \C$;
\item[(f)] $H^1(\BP(T_{\BP^m}), T(-1,-1)) \simeq \C$;
\item[(g)] $H^1(\BP^1\times \BP^1, T(-k, -2)) \simeq {\rm Sym}^{k-2} \C^2, k\geq 2 $;
\item[(h)] $H^1(\BP^1\times \BP^n, T(-k, -1)) \simeq {\rm Sym}^{k-2} \C^2 \otimes \C^{n+1},  k\geq 2.$
\end{itemize}
\end{itemize}
\end{theorem}

For any smooth projective variety $X$,   we have $H^q(X, T_X \otimes L^*)=0$ for all $q \geq 2$ provided that $K_X^* \otimes L^*$ is ample, by Akizuki-Nakano vanishing theorem. As a consequence, we have

\begin{corollary} \label{c.MS}
Assume $G/P$ satisfies $\clubsuit$. Let $L \in {\rm Pic}(G/P)$ be an ample line bundle such that $K_{G/P}^* \otimes L^*$ is ample. Then
$H^q(G/P, T_{G/P} \otimes L^*) = 0 $ for all $q \geq 0$.
\end{corollary}
\begin{proof}
The only case to be considered is $G/P \simeq \BP^1\times \BP^n$, then  $K_{G/P}^* = \0(2, n+1)$.  By the assumption that $K_{G/P}^* \otimes L^*$ is ample, we have $L=\0(1, a)$ with $a \leq n$. This implies that $H^1(G/P, T_{G/P} \otimes L^*)=0$ by Theorem \ref{t.MS}.
\end{proof}

Let $D_i \subset G/P$ be $r$ ample divisors and $X = \cap_{i=1}^r D_i$ their complete intersection. Assume that $X$ is smooth of the expected dimension, and irreducible.
Let $D:=\sum_i D_i$. Then we have the following Koszul exact sequence
\begin{equation}\label{e.Koszul}
0 \to \0_{G/P}(-D) \to \oplus_i \0_{G/P}(-D+D_i) \to  \cdots \to \oplus_i \0_{G/P}(-D_i) \to \sO_{G/P} \to \sO_X \to 0.
\end{equation}

The following fact is classical, see Lemma 5.7 in \cite{FH3}.

\begin{lemma}\label{l.vanish}
Let $0 \to \mathcal{F}_0 \to \mathcal{F}_1 \to \cdots \to
\mathcal{F}_m \to 0$ be an exact sequence of coherent sheaves on a
variety $X$. If $H^{q+j-1}(X, \mathcal{F}_{m-j})=0$ for
all $j \in \{1, 2, \cdots, m\}$, then $H^q(X, \mathcal{F}_m)=0$.
\end{lemma}

By \cite{D}, we may assume that $H^0(G/P, T_{G/P}) \simeq \fg$  up to representing $G/P$, if necessary, as another quotient $G'/P'$.
\begin{proposition} \label{p.Euler}
Assume $G/P$ satisfies $\clubsuit$ and $H^0(G/P, T_{G/P}) = \fg$.  Consider a smooth complete intersection $X=\cap_{i=1}^r D_i \subset G/P$  such that $K_{G/P}^* \otimes \0_{G/P}(-\sum_i D_i)$ is ample.  Then

$$
h^0(T_X) - h^1(T_X) =  \dim \fg  - \sum_{i=1}^r h^0(X, \0_{G/P}(D_i)|_X).
$$
\end{proposition}
\begin{proof}

Taking the tensor product of the Koszul exact sequence \eqref{e.Koszul} with $T_{G/P}$,
and  using Corollary \ref{c.MS} and Lemma \ref{l.vanish}, we get that
$$
H^0(X, T_{G/P}|_X)=   \fg \quad \mathrm{and}  \quad H^q(X, T_{G/P}|_X)=0 \; \forall q\geq 1.
$$

The exact sequence $0 \to T_X \to T_{G/P}|_X \to \oplus_{i=1}^r \0_{G/P}(D_i)|_X \to 0$ implies that
$$
0 \to H^0(T_X) \to H^0(X, T_{G/P}|_X) \to \oplus_{i=1}^r H^0(\sO_{G/P}(D_i)|_X ) \to
H^1(T_X) \to 0
$$
is exact, from which the claim follows.
\end{proof}

\begin{remark} \label{r.Fano}
By adjunction, we have $K_X^* = (K_{G/P}^*\otimes \0_{G/P}(-D))|_X$, which is ample by assumption, hence $X$ is Fano.  When $G/P$ is of Picard number one (a main case in our discussions), the converse also holds, namely if $X$ is Fano, then $K_{G/P}^*\otimes \0_{G/P}(-D)$ is ample on $G/P$. But in general, our assumption is stronger than the Fanoness of $X$. For example, take a general hypersurface $X$ of bidegree $(2,1)$ in $\BP^1 \times \BP^2$.  Then the map
$p: X \to \BP^2$ is a finite morphism (of degree 2). By adjunction, $K_X^* = \0(0,2)|_X = p^* \0_{\BP^2}(2)$ which is ample. Hence $X$ is Fano but $ \0(0,2)$ is not ample on $\BP^1 \times \BP^2$.
\end{remark}

\begin{lemma} \label{l.OX}
Let $X= \cap_{i=1}^r D_i \subset G/P$ be a smooth complete intersection such that $K_{G/P}^* \otimes \0_{G/P}(-\sum_i D_i)$ is ample.  Let $L_0\in {\rm Pic}(G/P)$ be the minimal ample line bundle.
Then
 $h^0(X,L_0|_X) = \dim V_P - s$, where $s=\sharp \{i| \0_{G/P}(D_i) \simeq L_0\}$.
\end{lemma}
\begin{proof}
Taking the tensor product of \eqref{e.Koszul} with $L_0$, we get
$$
0 \to \0_{G/P}(-D)\otimes L_0 \to \cdots \to \oplus_i \0_{G/P}(-D_i)\otimes L_0 \to L_0 \to L_0|_X \to 0.
$$
From Lemmas \ref{l.vanish} and \ref{l.O}, we deduce an exact sequence
 $$
 0\to \oplus_i H^0(\0_{G/P}(-D_i)\otimes L_0) \to H^0(L_0) \to H^0(L_0|_X) \to 0,
 $$
 which implies the claim since $L_0$ is the minimal ample line bundle on $G/P$.
\end{proof}

Note that $h^0(X,\0_{G/P}(D_i)|_X) \geq h^0(X, L_0|_X) = \dim V_P -s \geq \dim V_P -r$. Hence we obtain
\begin{corollary} \label{c.notrigid}
Assume $G/P$ satisfies $\clubsuit$ and $H^0(G/P, T_{G/P}) = \fg$. Let $X= \cap_{i=1}^r D_i \subset G/P$ be a smooth complete intersection of codimension $r$ such that $K_{G/P}^* \otimes \0_{G/P}(-\sum_i D_i)$ is ample. Then $X$ is not locally rigid if $\dim \fg  < r(\dim V_P -r)$.
\end{corollary}

%

From now on, we will assume further that $G$ is simple.

\begin{lemma}\label{l.dim}
Let $V$ be an irreducible representation of a simple Lie group $G$. Then

(1) $\dim V \neq \dim \fg +1$.

(2) $\dim V = \dim \fg$ if and only if $V$ is the adjoint representation.
\end{lemma}
\begin{proof}
Assume $G$ is of type $A_{\ell}$. The irreducible representations of $G$
of dimension $\leq (\ell+1)^2$ are classified in \cite{sk} (Proposition 7 on p.45) and the claim follows. For type $C_\ell$, we can apply \cite{sk} (Lemma 13 and Proposition 14 on p.50).  The case of  $\rm{SO}(m)$ follows from Proposition 20 in \cite{sk} (p. 54).  If $G$ is of exceptional type, we can apply Proposition 22 on p. 56 of \cite{sk}.
\end{proof}

%
%

\begin{remark}
Note that if $G$ is not simple, then there are exceptions.  For example, take $G/P = {\rm Gr}(2, 5) \times \BP^3$, then $\dim G=39$ and $\dim V_P=40$
\end{remark}

\begin{proposition} \label{p.bigV}
Let $G/P$ be a rational homogeneous variety with $G$ simple such that $H^0(G/P, T_{G/P})=\fg.$  Assume that $\dim \fg < \dim V_P$. Consider a smooth complete intersection $X = \cap_{i=1}^r D_i$ in $G/P$ such that $K_{G/P}^* \otimes \0_{G/P}(-\sum_i D_i)$ is ample. Then $X$ is not locally rigid.
\end{proposition}
\begin{proof}
Note that the condition $\dim \fg < \dim V_P$ implies that $G/P$  satisfies $\clubsuit$,
then by Lemma \ref{l.dim}, the assumption $\dim \fg < \dim V_P$ implies that $\dim V_P \geq \dim \fg +2$. Now the claim follows from Corollary \ref{c.notrigid}, as $r < \dim G/P < \frac{1}{2} \dim \fg$.
\end{proof}

By Proposition \ref{p.bigV}, we are reduced to the case $\dim \fg \geq \dim V_P$. By Lemma \ref{l.dim}, the case of equality implies that $V_P$ is the adjoint representation and then  $G/P$ is the adjoint variety.
In this case, $G/P$ has Picard number one except for type $A$, where $G/P = \BP(T_{\BP^m})$.
\begin{proposition} \label{p.adjoint}
Let $X = \cap_{i=1}^r D_i \subset  Z:=\BP(T_{\BP^m})$ be a general complete intersection of $r$ ample divisors.
 Assume that $K_Z^* \otimes \0_Z(-\sum_{i=1}^r D_i)$ is ample. Then $X$ is locally rigid if and only if
$X$ is a general hypersurface of bidegree $(1,1)$ of $\BP(T_{\BP^2})$. In this case, $X$ is isomorphic to the blowup of $\BP^2$ at 3 general points.
\end{proposition}
 \begin{proof}
Note that  $K_{\BP(T_{\BP^m})}^* = \sO(m,m)$, hence $r <m$ by the ampleness of $K_Z^* \otimes \0_Z(-\sum_{i=1}^r D_i)$.
 By a similar argument as in the proof of Proposition \ref{p.Euler}, we have
 $$h^0(T_X) - h^1(T_X) =  \dim \fg + s  - \sum_{i=1}^r h^0(X, \0_{Z}(D_i)|_X),$$
  where $s$ is the number $\sharp \{j | D_j\ \text{is of bidegree}\ (1,1)\}$.
 By Lemma \ref{l.OX}, we have $$h^0(X, \0_{Z}(1,1)|_X) = h^0(Z, \0_Z(1,1))-s = m^2+2m-s.$$  This gives that
 $h^0(T_X) - h^1(T_X) \leq m^2+2m+s -r(m^2+2m-s)$, which is negative if $r \geq 2$.

 Now assume $X \subset Z$ is a hypersurface of bidegree $(a, b)$. If $(a, b) \neq (1,1)$, we have $h^0(T_X) - h^1(T_X) =  \dim \fg - h^0(X, \0_{Z}(a,b)|_X)$.
 As $$h^0(X, \0_{Z}(a,b)|_X) \geq h^0(X, \0_{Z}(2,1)|_X) \geq h^0(Z, \0_Z(2,1))-1 = \frac{(m+1)(m^2+2m-2)}{2}-1,$$ we obtain that $h^0(T_X) - h^1(T_X) <0$, which implies that $X$ is not locally rigid.
 If $X \subset Z$ is of bidegree $(1,1)$, then $h^0(T_X) - h^1(T_X) = m^2+2m+1 - ( m^2+2m-1) = 2$. This implies that $X$ is locally rigid if and only if $h^0(T_X)=2$. On the other hand, for the adjoint action of
 $G={\rm PGL}_{m+1}$ on $\fg=\mathfrak{sl}_{m+1}$ , its stabilizer at a general point has dimension  $m$, hence $h^0(T_X) \geq m$, which gives $m \leq 2$.

 Let $X \subset \BP(T_{\BP^2})$ be a general hypersurface of bidegree $(1,1)$. The projection $X \to \BP^2$ is birational, with three fibers isomorphic to $\BP^1$, hence it is the blowup of $\BP^2$ along 3 points, which are in general position as $X$ is Fano.
 \end{proof}

 The following is probably well-known, but we do not find an explicit reference.
 \begin{proposition} \label{p.Pn}
 Let $X \subset \BP^N$ be a smooth Fano complete intersection. Then $X$ is locally rigid if and only if $X$ is isomorphic to a projective space or a hyperquadric.
 \end{proposition}
 \begin{proof}
 Let $(d_1, \cdots, d_r)$ be the multi-degree of $X$ such that $2 \leq d_1 \leq \cdots \leq d_r$. We may assume $\dim X \geq 2$ as the only Fano curve is $\BP^1$.
 If $X$ is not a hyperquadric, then $h^0(X, T_X)=0$ (see for example Lemma 7.3 \cite{FH3}).
 By a similar argument as that in Proposition \ref{p.Euler}, we have $$h^1(X, T_X) = \sum_{i=1}^r h^0(X, \0_X(d_i))-(N^2+2N) \geq r h^0(X, \0_X(2))-(N^2+2N),$$ while $h^0(X, \0_X(2)) = \binom{N+2}{2} - s$ with $s = \sharp\{j | d_j=2\}$.
 This implies that $h^1(X, T_X) >0$ if $r \geq 2$. Now if $X \subset \BP^N$ is a hypersurface of degree $d \geq 3$, then
 $h^1(X, T_X) = \binom{N+d}{d} - (N^2+2N) \geq \binom{N+3}{3} - (N^2+2N) >0$, which concludes the proof.
 \end{proof}

 In \cite{E}, irreducible representations $V$ of $G$ with $\dim G > \dim V$ are classified. It turns out that they are all fundamental representations (see Table 1 in the following section) except for $G/P=\BP^n$.
  By Proposition \ref{p.adjoint} and Proposition \ref{p.Pn}, we may assume from now on that $G$ is simple and  $G/P$ is of Picard number one satisfying $\clubsuit$.

\begin{remark} \label{r.sk}
When $G$ is semi-simple but not simple, a classification of the irreducible representations $V$ of $G$ such that $\dim G +1 \geq  \dim V$ is given in Section 3 of \cite{sk}(Note that in the notation therein, $G$ has 1 dimensional center, hence their classification gives all $V$ with $\dim G +1 \geq  \dim V$.)
With this, a similar result as Theorem \ref{t.main} can be obtained for any $G$ semi-simple. We leave this to the reader.
\end{remark}

\section{Rigidity of hypersurfaces in $G/P$}

In this section,  $G/P$ is a rational homogeneous variety of Picard number one.
Recall that for an irreducible representation $V$ of $G$,  there  exists an open subset $U$ such that the stationary subalgebras $\fg_v$ of all the points $v\in U$ are conjugate to a single subalgebra $\mathfrak h \subset \fg$ by \cite{Ri} (Theorem A).

The next table is taken from \cite{E} (table 1) and it gives all the
fundamental representations $V_P$ with $\dim V_P < \dim \fg$.
 In the column headed $\mathfrak h$ is given the generic stationary subalgebra. In those cases when $\mathfrak h$ is the direct sum of ideals $\mathfrak{h_1,\ldots,h}_k$, we write $\mathfrak h=\mathfrak{h_1\oplus\ldots\oplus h}_k$. If $\mathfrak h$ decomposes into the semidirect sum of a subalgebra $P$ and an ideal $U$, we write $\mathfrak h=P+U$ and in parentheses we specify the action of $P$ on $U$. Furthermore, $U_k$ is a $k$-dimensional commutative Lie algebra.

\[\begin{array}{|r|r|r|r|r|}

\hline
\textrm{type} & k & \dim V^{\lambda_k}&\mathfrak h&\dim \mathfrak h\\
\hline
A_\ell&1&\ell+1&A_{\ell-1}+U_\ell(R(\lambda_1))
&\ell^2+\ell-1\\
\hline
A_{2j-1}&2&j(2j-1)&C_j&2j^2+j\\
\hline
A_{2j}&2&j(2j+1)&C_j+U_{2j}(R(\lambda_1))&2j^2+3j\\
\hline
A_5&3&20&A_2\oplus A_{2}&16\\
\hline
A_6&3&35&G_2&14\\
\hline
A_7&3&56&A_2&8\\
\hline
B_\ell&1&2\ell+1&D_\ell&2\ell^2-\ell\\
\hline
B_3&3&8&G_2&14\\
\hline
B_4&4&16&B_3&21\\
\hline
B_5&5&32&A_4&24\\
\hline
B_6&6&64&A_2\oplus A_2&16\\
\hline
C_\ell&1&2\ell &C_{\ell-1}+U_{2\ell-1}(R(\lambda_1)+1)&2\ell^2-\ell \\
\hline
C_\ell&2&2\ell^2-\ell-1&\underbrace{A_1\oplus \ldots\oplus A_1}_\ell&3\ell\\
\hline
C_3&3&14&A_2&8\\
\hline
D_\ell&1&2\ell&B_{\ell-1}&2\ell^2-3\ell+1\\
\hline
D_5&4&16&B_3+U_8(R(\lambda_3))&29\\
\hline
D_6&5&32&A_5&35\\
\hline
D_7&6&64&G_2\oplus G_2&28\\
\hline
G_2&1&7&A_2&8\\
\hline
F_4&4&26&D_4&28\\
\hline
E_6&1&27&F_4&52\\
\hline
E_7&7&56&E_6&78\\
\hline
\end{array}
\]
\begin{center}
Table 1: \`{E}la\v{s}vili's list
\end{center}

\begin{lemma} \label{l.hyper}
Assume $G/P$ satisfies $\clubsuit$ and $H^0(G/P, T_{G/P}) =\fg$.  Let $X \subset G/P$ be a smooth Fano hypersurface of degree $d\geq 2$. Then $X$ is not locally rigid.
\end{lemma}
\begin{proof}
We may assume $\dim \fg \geq \dim V_P$ by Proposition \ref{p.bigV}.   By Proposition \ref{p.Euler}, we have
$h^0(X, T_X) - h^1(X, T_X) = \dim \fg - h^0(X, \sO_X(d))$ while $h^0(X, \sO_X(d)) = h^0(G/P, \sO_{G/P}(d)) -1$.
Thus if $h^0(G/P, \sO_{G/P}(d)) > \dim \fg +1$, then $H^1(X, T_X) \neq 0$, which implies that $X$ is not locally rigid.

By Lemma \ref{l.dim}, if $\dim \fg = \dim V_P$, then $V_P$ is the adjoint representation and $G/P$ is the adjoint variety.
Then $h^0(G/P, \sO_{G/P}(d)) > \dim \fg +1$ if $d \geq 2$.

Finally, in \cite{E}, all irreducible representations $V$ of $G$ with $\dim \fg > \dim V$ are listed, from which we deduce that $\dim \fg > \dim h^0(G/P, \sO_{G/P}(d))$ is only possible for $G/P \simeq \BP^n$ and $d=2$, which is excluded by our assumption $\clubsuit$.
\end{proof}

\begin{proposition} \label{p.aut}
Assume $G/P \subset \BP(V_P^*)$ satisfies $\clubsuit$ and $H^0(G/P, T_{G/P}) = \fg$.  Let $L \subset \BP(V_P^*)$ be a linear subspace of codimension $r$ such that
 $X = G/P \cap L$ is  smooth Fano.
 We denote by $\mathfrak{aut}(G/P, L)$ the Lie algebra of automorphisms of $G/P$ preserving the linear space $L$.
 Then
 \begin{itemize}
\item[(i)] $H^0(X, T_X) \simeq \mathfrak{aut}(G/P, L)$.
\item[(ii)] $X$ is locally rigid if and only if the $G$-orbit of $[L^\perp] \in {\rm Gr}(r, V_P)$ is open.
\end{itemize}
\end{proposition}
\begin{proof}
(i) By Lemma \ref{l.OX}, we have $H^0(X, \0_X(1)) = L$, hence $L$ is the linear span of $X$.
By the proof of Proposition \ref{p.Euler},  we have $H^0(X,  T_{G/P}|_X ) \simeq H^0(G/P, T_{G/P}) = \fg$.  By the normal bundle exact sequence
$0 \to T_X \to  T_{G/P}|_X \to \mathcal{N}_{X|G/P} \to 0$, we get that
$$H^0(X, T_X) = {\rm Ker} (H^0(X, T_{G/P}|_X) \simeq  H^0(G/P, T_{G/P}) \to H^0(X,  \mathcal{N}_{X|G/P})),$$
namely $H^0(X, T_X)$ identifies with the set of vector fields on $G/P$ which preserves $X$ (hence its linear span $L$).

(ii) By Proposition \ref{p.Euler}, we have
$$h^1(X, T_X) = r(\dim V_P-r) - (\dim \fg - \dim \mathfrak{aut}(G/P, L)) = \dim {\rm Gr}(r, V_P) - \dim G \cdot [L^\perp], $$ which vanishes if and only if  the $G$-orbit of $[L^\perp] \in  {\rm Gr}(r, V_P)$ is open.
\end{proof}

In the rest of this section, we will only consider hyperplane sections of $G/P$.
Let $L \subset \BP V_P^*$ be a general hyperplane, projectivization of the affine hyperplane
$\widehat{L} \subset V_P^*$.
 Let $X = G/P \cap L$ be the corresponding hyperplane section of $G/P$.
We will denote by ${\bf l} \subset V_P$ the line orthogonal to the hyperplane $\widehat{L} \subset V_P^*$.
Recall that $V_P^*$ is an irreducible representation of $G$ and so is $V_P$.  We denote by $G_{\bf l}$ the subgroup of $G$ preserving ${\bf l}$ and $\fg_{\bf l}$ its Lie algebra.
For $v \in \bf{l}$ a non-zero element, the stabilizer $G_v$ is a subgroup of $G_{\bf{l}}$. The quotient $Q: = G_{\bf l}/G_v$ acts on $\bf{l}$ by a subgroup of $\C^*$.
\begin{lemma} \label{l.reductive}
If $G_v$ is reductive, then $Q$ is a finite group, hence $\fg_{\bf l} = \fg_v$.
\end{lemma}
\begin{proof}
As $L$ is a general hyperplane, the point $v \in \bf{l}$ is a general point of $V_P$. If $G_v$ is reductive, then by \cite{P}, the orbit $G \cdot v$ is closed in $V_P$. If $Q \simeq \C^*$, then $Q$ acts on $\bf{l}$ by scalars, hence $G \cdot v \supset Q \cdot v = \bf{l} \setminus \{0\}$. By the closedness of the orbit, we get that $0 \in G \cdot v$, which is absurd.
\end{proof}

\begin{proposition} \label{p.stab}
Assume $G/P$ satisfies $\clubsuit$ and $H^0(G/P, T_{G/P}) = \fg$. Let $v \in {\bf l}$ be a non-zero point and $\fg_v$ the Lie algebra of the stabilizer $G_v$.
Then
$$
\mathfrak{aut}(X) = \begin{cases} \fg_v \oplus \C & \mbox{if} \ G/P = {\rm Gr}(2, 2k+1)\ \mbox{or}\ \mathbb{S}_5, \\ \fg_v & \mbox{otherwise}. \end{cases}
$$
\end{proposition}
\begin{proof}
First note that $\mathfrak{aut}(G/P, L)$ is exactly $ \fg_{\bf l}$.  If $\dim V_P > \dim G$, then by \cite{AVE} (Corollary on p.260), the stabilizer $G_v$ is discrete, hence  by Lemma \ref{l.reductive}, we have $\fg_v = \fg_{\bf l} = 0$.
If $\dim V_P = \dim G$, then $V_P$ is the adjoint representation (cf. Lemma \ref{l.dim}) and in this case $\fg_v$ is a Cartan subalgebra, hence by Lemma \ref{l.reductive}, we have $\fg_v = \fg_{\bf l}$.

Now assume $\dim V_P < \dim G$, then the stabilizer $\fg_v$ is computed in \cite{E} (table 1). One checks that $\fg_v$ is not reductive only for $G/P = {\rm Gr}(2, 2k+1)$ or $\mathbb{S}_5$.

When $G/P = {\rm Gr}(2, 2k+1)$, then $X$ is the so-called odd symplectic Grassmanian. Its automorphism group is computed in \cite{PV}, from which one checks that $\mathfrak{aut}(X) \simeq \fg_v \oplus \C$.  When $G/P = \mathbb{S}_5$, its Lie algebra of automorphism group is well-known (see for example Proposition 3.9 \cite{FH1}) and one checks directly the claim.
\end{proof}

By Propositions \ref{p.Euler} and \ref{p.stab}, we obtain the following
\begin{corollary} \label{c.h1}
Assume $G/P$ satisfies $\clubsuit$ and $H^0(G/P, T_{G/P}) = \fg$.
Let $X =  G/P \cap L$ be a general hyperplane section and $v\in \bf{l}$ a nonzero point.  Then
 $$h^1(X, T_X) = \begin{cases} 
 \dim \fg_v + \dim V_P -\dim \fg  & \mbox{if}\ G/P = {\rm Gr}(2, 2k+1)\ \mbox{or}\ \mathbb{S}_5, \\
\dim \fg_v + \dim V_P -\dim \fg -1 & \mbox{otherwise}.\end{cases} $$
\end{corollary}

%

\begin{lemma}\label{l.SymGrass}
Let $X \subset {\rm Gr}(2, n+1)  (n\geq 4)$ be a general codimension 2 linear section, then $X$ is locally rigid if and only if either $n$ is even or $n=5$.
\end{lemma}
\begin{proof}
By Proposition \ref{p.Euler}, we have $h^0(T_X) - h^1(T_X) = (n^2+2n) -2(n(n+1)/2-2) = n+4$.
By \cite{PV}, we have
$$\dim {\rm Aut}(X) = \begin{cases} n+4  &  n \ \mbox{even} \\  3(n+1)/2 & n \ \mbox{odd}. \end{cases} $$
The claim follows immediately.
\end{proof}

\begin{remark}
By \cite{PV}, the general codimension 2 linear section $X \subset {\rm Gr}(2, 2k+1)$ is quasi-homogeneous if and only if $k \leq 3$.
In even dimension, a codimension 2 linear section $X \subset {\rm Gr}(2, 2k)$ is defined
by a pencil of skew-symmetric forms, and those that are not of maximal rank define in
general a $k$-tuple of points on $\pit^1$, well-defined up to ${\rm PGL}_2$. This $k$-tuple
of points has the same number of moduli as $X$.
Moreover it is easy to see that $X$
is quasi-homogeneous only when $k\le 3$. For $k=3$ it is a compactification of
${\rm SL}_2\times {\rm SL}_2\times {\rm SL}_2/{\rm diag}({\rm SL}_2)$.
\end{remark}

\begin{remark} \label{r.GL(25)}
Linear sections of ${\rm Gr}(2,5)$ have been studied classically  and appear in the
classification of del Pezzo manifolds. It is well-known that there is a unique isomorphism
class of del Pezzo manifolds of degree five in each dimension between 2 and 6, hence they are all locally rigid.
\end{remark}

\begin{theorem} \label{t.rigid}
Let $G/P$ be a rational homogeneous variety of Picard number one.
Let $X \subset G/P$ be a general hyperplane.  Then $X$ is locally rigid if and only if $G/P$ is isomorphic to one of the following
$$
\mathbb{P}^n,  \mathbb{Q}^n, {\rm Gr}(2, n), {\rm Gr}(3, 6), {\rm Gr}(3,7), {\rm Gr}(3,8), $$
$$ \BS_5, \BS_6, \BS_7,  {\rm Gr}_\omega(2, 6), {\rm Lag}(3, 6), F_4/P_4, E_6/P_1, E_7/P_7.$$
\end{theorem}
\begin{proof}
If $G/P$ does not satisfy $\clubsuit$,  we  need to consider two cases: $C_\ell/P_2$ and $F_4/P_4$.
A general hyperplane section  of $C_\ell/P_2$ is a codimension 2 linear sections of $A_{2\ell-1}/P_2$,
which is locally rigid if and only if  $\ell = 3$ by Lemma \ref{l.SymGrass}. For $F_4/P_4$, its general hyperplane section $X$ is a codimension 2 linear section of $E_6/P_1$. By Proposition \ref{p.Euler}, we have $h^0(T_X) - h^1(T_X) =28$.  By Proposition 48 (\cite{sk}, p. 139), the stabilizer ${\mathfrak aut}(E_6/P_1, L)$ is $\mathfrak{so}(8)$, which implies $h^0(X, T_X) = 28$ by Proposition \ref{p.aut}, hence $X$ is locally rigid.

Now we assume $G/P$ satisfies $\clubsuit$.
By Proposition \ref{p.bigV}, we may assume $\dim V_P \leq \dim \fg$.
If $\dim V_P = \dim G$, then $V_P$ is the adjoint representation  by Lemma \ref{l.dim}.
 As a consequence, $\dim \fg_v = {\rm rk}(\fg)$. By Corollary \ref{c.h1},
$\dim H^1(X, T_X) = {\rm rk}(\fg) -1$, which is non-zero except for type $A_1$.

Now assume $\dim V_P < \dim G$,  then the stabilizer $\fg_v$ is computed in Table 1. Then a case-by-case check using Corollary \ref{c.h1} concludes the proof.
\end{proof}

As an application, we recover the following well-known fact from \cite{A}.
\begin{corollary} \label{c.homogeneous}
A general hyperplane section of $G/P$ of Picard number one is homogeneous if and only if $G/P$ is isomorphic to one of the following
$$
\mathbb{P}^n,  \mathbb{Q}^n, {\rm Gr}(2, 2k), E_6/P_1.
$$
\end{corollary}

\section{Rigidity of complete intersections in $G/P$}

\subsection{End of the classification}  \quad \\

\vspace{-3mm}
Let $G/P$ be a rational homogeneous variety of Picard number one.
Let $X \subset G/P$ be a smooth complete intersection of codimension $r \geq 2$. We may assume $\dim X \geq 2$.
By Corollary \ref{c.notrigid},  if $X$ is locally rigid, then $\dim \fg  \geq  r(\dim V_P -r)$. The following lists all such possibilities (by using Table 1).

\begin{lemma}
Assume $G/P$ satisfies $\clubsuit$ and $H^0(G/P, T_{G/P}) = \fg$. Then  $\dim \fg  \geq  r(\dim V_P -r)$ for some $r \geq 2$ holds only for the following cases:
\begin{itemize}
\item[(1)] $G/P = A_\ell/P_2$ and $r=2$, or $\ell=4$ and $r=3, 4$;
\item[(2)] $G/P = \BS_5$ and $r=2, 3$;
\item[(3)] $G/P = \BS_6$ and $r=2$;
\item[(4)] $G/P= E_6/P_1$ and $r=2, 3$;
\item[(5)] $G/P = E_7/P_7$ and $r=2$.
\end{itemize}
\end{lemma}

By a similar argument as in Lemma \ref{l.hyper}, the only possible complete intersections which are locally rigid in these cases are linear sections.
Case (1) is done by Lemma \ref{l.SymGrass} and Remark \ref{r.GL(25)}.
 We will consider case (2) in the following subsection. Concerning case (4), the codimension 2 linear section of $E_6/P_1$ is a hyperplane section of $F_4/P_4$ and has been studied in the previous section.  The remaining three cases are treated by the following result.
 Alltogether, this will complete the proof of Theorem \ref{t.main}.

\begin{proposition}
A general codimension 2 (resp. 3) linear section of $\BS_6$ or $E_7/P_7$  (resp. $E_6/P_1$) is not locally rigid.
\end{proposition}
\begin{proof}
For $X$ a smooth codimension 2 (resp. 3) linear section of $\BS_6$ or $E_7/P_7$ (resp. $E_6/P_1$), we have
$h^0(T_X) - h^1(T_X)  = 6, 25$ (resp. 6) by Proposition \ref{p.Euler}. On the other hand, for $X$ general, the stabilizer of its linear span in $\BP V_P^*$ is of type
$3A_1, D_4$ (resp. $A_2$) by \cite{V} (table on p. 491-492), so $H^0(T_X)$ has dimension $9, 28$ (resp. 8) by Proposition \ref{p.aut}.
This gives that $h^1(T_X)=3, 3$ (resp. $2$), concluding the proof.
\end{proof}

\subsection{Sections of the 10-dimensional spinor variety} \quad \\

\vspace{-3mm}
Let $S = \mathbb{S}_5 \subset \pit^{15}$.  For $k \geq 1$, we denote by  $S_k \subset \pit^{15-k}$
a smooth linear section of $S$ of codimension $k$. The hyperplane section $S_1$ is a horospherical variety
with Picard number one and non reductive automorphism group, it appears as case 2 of \cite{Pa}, Theorem 0.1.
It is uniquely defined up to isomorphism, but this is no longer the case of $S_k$ for $k=2,3$.

\begin{lemma} \label{l.S5}
 (1) We have $h^0(T_{S_2}) - h^1(T_{S_2}) = 17$ and $h^0(T_{S_3}) - h^1(T_{S_3}) = 6$.

 (2) The general $S_2$ and $S_3$ are locally rigid.
\end{lemma}
\begin{proof}
Claim (1) is immediate from Proposition \ref{p.Euler}.  Let $k=2, 3$. The action of ${\rm GL}_k\times {\rm Spin}_{10}$ on
$\mathbb{C}^k\otimes V_{16}$ (where $V_{16}$ is a
spin representation) is known to be quasi-homogeneous by \cite{sk} (Propositions 32 and 33, p. 124-126), and therefore the action of ${\rm Spin}_{10}$
on the Grassmannian ${\rm Gr}(k,V_{16})$ is also quasi-homogeneous. By Proposition \ref{p.aut},
the general $S_2$ and $S_3$ are locally rigid.
\end{proof}

In the following, we will study the quasi-homogeneity of the sections $S_k (k=2, 3)$, and show that it does
not always imply their local rigidity.  To that purpose, we first introduce some results which allow us to determine $\mathfrak{aut}(S_k)$.

For a smooth projective subvariety $Z \subset \BP^N$ covered by lines, the variety of lines on $Z$ through a point $z \in Z$ is called the VMRT of $Z$ at $z$.
For equivariant compactifications of affine spaces, we can describe the
infinitesimal automorphisms in terms of prolongations of the VMRT. For that purpose,
we recall the following

\begin{definition}
Let $V$ be  a complex vector space and $\fg \subset {\rm End}(V)$
a Lie subalgebra. The {\em $k$-th prolongation} (denoted by
$\fg^{(k)}$) of $\fg$ is the space of symmetric multi-linear
homomorphisms $A: \Sym^{k+1}V \to V$ such that for any fixed $v_1,
\cdots, v_k \in V$, the endomorphism $A_{v_1, \ldots, v_k}: V \to
V$ defined by $$v\in V \mapsto A_{v_1, \ldots, v_k, v} := A(v,
v_1, \cdots, v_k) \in V$$ is in $\fg$. In other words, $\fg^{(k)}
= \Hom(\Sym^{k+1}V, V) \cap \Hom(\Sym^kV, \fg)$.
\end{definition}

It is shown in \cite{HM} that for a smooth non-degenerate $C \subsetneq \BP^{n-1}$, the second prolongation satisfies $\aut(\hat{C})^{(2)}=0$, where
$\aut(\hat{C}) \subset \mathfrak{gl}(n)$ is the Lie algebra of infinitesimal automorphisms of $\hat{C}$.
The following result is a combination of Propositions 5.10, 5.14 and 6.13 in \cite{FH1}.

\begin{proposition} \label{p.prolong}
Let $X$ be an $n$-dimensional smooth uniruled projective variety of Picard number one. Assume that the  VMRT   at a general point is isomorphic to
a smooth irreducible non-degenerate projective subvariety $C \subsetneq \BP^{n-1}$. Then
$$
\dim \mathfrak{aut}(X) \leq n + \dim \mathfrak{aut}(\hat{C})+\dim \mathfrak{aut}(\hat{C})^{(1)},
$$
with equality if and only if $X$ is an equivariant compactification of $\C^n$.    In case of equality, we have
$\aut(X) \simeq \mathbb{C}^n \sd \aut(\hat{C}) \oplus \aut(\hat{C})^{(1)}$.

\end{proposition}

\begin{lemma} \label{l.prolong}
Let $C \subset {\rm Gr}(2, 5)$ be a general codimension 2 linear section. Then $\mathfrak{aut}(\hat{C})^{(1)} \simeq \C$.
\end{lemma}
\begin{proof}
First notice that $C \subset \BP^7$ is quadratically symmetric by Proposition 7.6 \cite{FH2}, hence $\mathfrak{aut}(\hat{C})^{(1)} \neq 0$ by Proposition 7.11 \cite{FH2}. By the proof of Theorem 6.15 \cite{FH2} (whose conclusion is not correct, since there is an error in the proof of Proposition 2.9 {\em loc. cit.}),
we get that $\mathfrak{aut}(\hat{C})^{(1)} \simeq \C$ as the VMRT of $C$ (namely a twisted cubic in $\BP^3$) has no prolongations.
\end{proof}

By Corollary 6.17 \cite{K}, there are exactly two isomorphic classes of smooth codimension 2 linear sections of $\BS_5$. By Remark 2.13 \cite{FH2}, the special section is an equivariant compactification of $\C^8$ while the general one is not.

\begin{proposition}
 If $S_2$ is special, then $h^0(T_{S_2}) = 18$ and $h^1(T_{S_2}) = 1$.
\end{proposition}

\begin{proof}
The VMRT $C$ of $S_2$ is a codimension 2 linear section of ${\rm Gr}(2, 5)$, so $\aut(C)$ has dimension 8 by \cite{PV}.
If $S_2$ is special, then it is an equivariant compactification of $\C^8$, hence by Proposition \ref{p.prolong}, we have
 $\aut(S_2) \simeq \mathbb{C}^8 \sd \aut(\hat{C}) \oplus \aut(\hat{C})^{(1)}$.
Since $\aut(\hat{C_1})^{(1)}$ is one-dimensional by Lemma \ref{l.prolong}, we obtain that $\dim  \aut(S_2) = 18$.
As  $H^0(\0_{S_2}(1))$ has dimension 14, this gives that $H^1(T_{S_2}) = \cit$, proving (1).

\end{proof}

The general section $S_2$ has a very different automorphism group from that of the special one. By
\cite{sk} (Propositions 32, p. 124), the former is of type $G_2\times SL_2$. This can be understood from
the following construction: recall that if $n=p+q$, a half-spin representation of $\fspin_{2n}$, when restricted
to the subalgebra $\fspin_{2p+1}\times \fspin_{2q-1}$, is isomorphic to the tensor product of the spin
representations of $\fspin_{2p+1}$ and $\fspin_{2q-1}$. In particular, for $n=5$, $p=1$, $q=4$, the
half-spin representation $\Delta_{10}$ of $\fspin_{10}$ restricts to $\Delta_{3}\otimes \Delta_{7}$.
Of course $\fspin_{3}\simeq\fsl_2$ and its spin representation $\Delta_{3}$ is just the natural representation
$V_2$ of $\fsl_2$. Now take another copy of $\fsl_2$ with its natural representation $U_2$, and consider
the representation $U_2\otimes \Delta_{10}$ of $\fsl_2\times \fspin_{10}$. When restricted to
$\fsl_2\times \fspin_{3}\times \fspin_{7}\simeq \fsl_2\times \fsl_2\times \fspin_{7}$, and then
to the subalgebra $\delta(\fsl_2)\times\fg_2$, where $\delta(\fsl_2)\subset \fsl_2\times \fsl_2$
denotes the diagonal subalgebra, one gets
$$U_2\otimes \Delta_{10}\simeq U_2\otimes U_2\otimes (\CC\oplus V_7),$$
where $V_7$ is the natural representation of $\fg_2\subset  \fspin_{7}$. In particular the line
$L=\wedge^2U_2\subset U_2\otimes \Delta_{10}$ is fixed by $\delta(\fsl_2)\times\fg_2$.

\begin{lemma}
The stabilizer of $L$ in $\fsl_2\times\fspin_{10}$ is exactly $\delta(\fsl_2)\times\fg_2$.
\end{lemma}

\begin{proof}
Let $(u_1,u_2)$ and $(v_1,v_2)$ be basis of $U_2$ and $V_2$, respectively. Since $\fg_2\subset
\fspin_{7}$ is the stabilizer of a generic point $p$ in the spin representation $\Delta_7$, we may
suppose that the line $L$ is generated by $q=(u_1\otimes v_2-u_2\otimes v_1)\otimes p$. Let us
compute the stabilizer of $q$ in $\fsl_2\times \fspin_{10}$. We can decompose
$$\fspin_{10}=\fsl_2\times\fspin_{7}\oplus {\rm End}_0(V_2)\otimes V_7.$$
Consider in ${\rm End}_0(V_2)$ the standard basis $(\theta_1, \theta_2, \theta_3)=
(v_1^*\otimes v_1-v_2^*\otimes v_2, v_2^*\otimes v_1, v_1^*\otimes v_2)$. We write
an element of $\fsl_2\otimes\fspin_{10}$ as $M = X+Y+Z+\theta_1\otimes w_1+ \theta_2\otimes w_2+
 \theta_3\otimes w_3$ for $X,Y\in \fsl_2$, $Z\in \fspin_7$ and $w_1, w_2, w_3\in V_7$.
The condition that $Mq=0$ is equivalent to the equations $Zp=0$ and
$$w_1\ast p= (X_{11}+Y_{22})p, \quad w_2\ast p= (X_{21}-Y_{21})p, \quad w_3\ast p= (X_{12}-Y_{12})p,$$
 where we denote by $\ast$ the Clifford multiplication  map from $V_7\otimes\Delta_7$ to $\Delta_7$.
In terms of the Cayley algebra $\OO$, we can identify $\Delta_7$ with $\OO$, $p$ with the unit octonion
and $V_7$ with ${\rm Im}(\OO)$, and $\ast$ is then just the octonionic multiplication. In particular, $w*p$ identifies
with $w$, and can never a be a non zero multiple of $p$.  The previous equations therefore reduce to
$w_1=w_2=w_3=0$, $X=Y$, and $Zp=0$, that is, $Z$ must belong to $\fg_2$.
\end{proof}

As a consequence, the ${\rm SL}_2\times {\rm Spin}_{10}$-orbit of $L$ in $\PP(U_2\otimes \Delta_{10})$ is open,
and so must be the orbit of the corresponding plane of $\Delta_{10}$. With the notations we have
just used, this plane is nothing else than $V_2\otimes p\subset V_2\otimes  \Delta_7=\Delta_{10}$.
We identify this subspace with its orthogonal ($V_2$ and $\Delta_7$ are naturally self-dual), and
we aim at describing the corresponding linear section $S_2$ of the spinor variety. For this we use the fact
that the spinor variety is defined by its quadratic equations, which are parametrized by
$V_{10}={\rm Sym}^2V_2\oplus V_7$.
These equations can be understood as follows: we have
$${\rm Sym}^2(V_2\otimes  \Delta_7)={\rm Sym}^2V_2\otimes {\rm Sym}^2 \Delta_7 \oplus
\wedge^2V_2\otimes \wedge^2 \Delta_7.$$
Note that ${\rm Sym}^2 \Delta_7$ has a unique invariant  (up to scalar), with an irreducible complement,
while $\wedge^2 \Delta_7=\fspin_7\oplus V_7$. This means that an element $v_1\otimes p_1+
v_2\otimes p_2$ of $\Delta_{10}$ belongs to (the cone over) the spinor variety if and only if
$$ Q(p_1)=Q(p_2)=Q(p_1,p_2)=0 \qquad \mathrm{and} \qquad \Omega(p_1,p_2)=0,$$
where $Q$ is the unique invariant quandratic form on $\Delta_7$, and $\Omega : \wedge^2 \Delta_7
\longrightarrow V_7$ the unique invariant map (up to scalar). Now we restrict to the codimension two linear
section $S_2$ orthogonal to $V_2\otimes p$, which just means that $p_1, p_2$ must be orthogonal to $p$.
Recall that as $\fg_2$-modules, $\Delta_7\simeq \CC p\oplus V_7$. Moreover, if we identify $V_7$
with the space of imaginary octonions, the restriction of $Q$ must be (a multiple of) the standard
quadratic form, and the unique $\fg_2$-invariant map from $\wedge^2 V_7$ to $V_7$ is given by
the imaginary part of the octonionic multiplication. We conclude that $v_1\otimes p_1+
v_2\otimes p_2$ belongs to (the cone over) $S_2$ if and only if, either $p_1$ and $p_2$ are colinear
and of norm zero, or they generate what is called a {\it null plane} in \cite{LM}, that is, a plane of imaginary
octonions in restriction to which the octonionic product vanishes identically. Since $G_2$ acts
transitively on the space of null-planes, we can finally conclude that $S_2$ is quasi-homogeneous
under the action of $G_2\times {\rm SL}_2$. Note moreover that it follows from this explicit description that
$S_2$ is isomorphic with the two-orbits variety denoted $X_2$ in \cite{Pa} (Definition 2.12).
Taking into account Theorem 0.2 of \cite{Pa}, we summarize our discussion as follows:

\begin{proposition} \label{p.2S5}
The general codimension two linear section $S_2$ of the spinor variety  $S = \mathbb{S}_5 \subset \pit^{15}$
is a two-orbits variety, which is quasi-homogeneous under its connected automorphism group $G_2\times {\rm PSL}_2$.
\end{proposition}

\begin{remark}
Here is another way to prove that $X_2 \simeq S_2$: by \cite{Pa} (Section 2.2.1), the connected automorphism group ${\rm Aut}^0(X_2)$ is isomorphic to $G_2\times {\rm PSL}_2$, which acts on $X_2$ with two orbits.  The closed orbit $Y$ is isomorphic to $\mathbb{Q}^5 \times \BP^1$ and the normal bundle
$\mathcal{N}_{Y|X_2}$ is isomorphic to the vector bundle of rank 2 on $Y$ associated to the irreducible representation with weights $\lambda_2 - \lambda_1 + 2 \lambda_0$ and $-\lambda_2 + 2\lambda_1 + 2\lambda_0$ (where $\lambda_1, \lambda_2$ are fundamental weights of $G_2$ and $\lambda_0$ is that of ${\rm PSL}_2$.) Note that in \cite{Pa}, he used the convention $\omega_i$ in stead of $\lambda_i$, which leads to a sign change.
This gives that the first Chern class $c_1(\mathcal{N}_{Y|X_2}) = L^{\lambda_1 + 4\lambda_0} = \0_Y(1,4).$
By adjunction formula, we have $K^*_{X_2}|_Y = K^*_Y \otimes c_1(\mathcal{N}_{Y|X_2}) = \0_Y(6,6)$, which implies that $K_{X_2}^* = \0(6)$.  This shows that $X_2$ is a Mukai variety, hence it must be a codimension 2 linear section of $\mathbb{S}_5$. As its automorphism group has dimension 17, it must be the general codimension 2 linear section $S_2$.
\end{remark}

\medskip
Now we consider $S_3$. From the classification results in \cite{weymanE8} (page 40), one can easily
check that there exists exactly four isomorphism classes.
By Remark 2.13 \cite{FH2}, although this is not the case of the general one, it can happen
that $S_3$ is an equivariant compactification of $\C^7$; in this case
we will say it is very special.

\begin{proposition}
If $S_3$ is very special, then  $h^0(T_{S_3})=11$ and  $h^1(T_{S_3}) = 5$.
%
\end{proposition}

\begin{proof}
If  $S_3$ is an equivariant compactification of $\cit^7$, its VMRT is a codimension 3 linear section of ${\rm Gr}(2, 5)$, whose automorphism group is ${\rm PGL}_2$. Hence $h^0(T_{S_3}) =11$
by Proposition \ref{p.prolong}, and the claim follows from Lemma \ref{l.S5}.
%
\end{proof}

In particular its automorphism group acts on the very special $S_3$ quasi-homogeneous\-ly, while the
automorphism group of the general $S_3$, which is of type $SL_2\times SL_2$ (\cite{sk}, Proposition 33 p. 126),
is too small for that.

\section{Quasi-homogeneous hyperplane sections}

Let $G/P$ be a rational homogeneous variety of Picard number one and $X \subset G/P$ a general hyperplane section. We consider the following question: when is $X$ quasi-homogeneous, in the
sense that ${\rm Aut}(X)$ acts on $X$ with an open orbit?

If $\dim V_P \geq \dim \fg$, then $X$ cannot be quasi-homogeneous, because $\dim \mathfrak{aut}(X)$ is smaller than $\dim X$.  When $\dim V_P < \dim \fg$, we are in the list of Table 1 and we do a case-by-case check.

(i) If $G/P = {\rm Gr}(2, n)$, then $X$ is either homogeneous (for $n$ even) or it is an odd symplectic Grassmanian, which is quasi-homogeneous.

(ii) If $G/P = {\rm Gr}(3, 6)$ (resp. ${\rm Lag}(3, 6), \mathbb{S}_6$,  $E_7/P_7$), then by \cite{R} (Theorem 3), the connected automorphism group ${\rm Aut}^0(X)$ is isomorphic to ${\rm SL}_3 \times {\rm SL}_3 $ (resp. ${\rm SL}_3$, ${\rm SL}_6$, $E_6$), which acts on $X$ with an open orbit isomorphic to
 ${\rm SL}_3 \times {\rm SL}_3 /{\rm diag}({\rm SL}_3)$ (resp. ${\rm SL}_3/{\rm SO}_3$, ${\rm SL}_6/{\rm Sp}_6$, $E_6/F_4$), hence it is quasi-homogeneous.

(iii) For $G/P=\mathbb{S}_5=D_5/P_5$, $X$ is well-known to be quasi-homogeneous. In fact, it is one of the two-orbits varieties studied in \cite{PP}.

(iv) If $G/P = C_\ell/P_2$, then $G/P$ is a general hyperplane section of ${\rm Gr}(2, 2\ell)$, hence $X$ is a general codimension 2 linear section of ${\rm Gr}(2, 2\ell)$. By \cite{PV}, $X$ is quasi-homogeneous if and only if $\ell\le 3$.

(v) For $G/P=G_2/P_1, E_6/P_1$, the general hyperplane sections are homogeneous.

(vi) For $G/P={\rm Gr}(3,8)$, then $\mathfrak{aut}(X) = \mathfrak{sl}_3$ has too small dimension
for $X$ to be quasi-homogeneous.

The remaining cases are ${\rm Gr}(3,7)$, $\BS_7$ and $F_4/P_4$. In the following subsections, we will prove that their general hyperplane sections are quasi-homogeneous. This will conclude the proof of Theorem \ref{t.main1}.

\subsection{${\rm Gr}(3,7)$} \quad \\

\vspace{-3mm}
Let $V_7$ be a seven-dimensional vector space. The stabilizer in ${\rm SL}(V_7)$ of a generic three-form
$\omega\in\wedge^3V_7^*$ is a subgroup isomorphic to $G_2$ \cite{sk} (Proposition 8, p. 86). This
can be understood by letting $V_7={\rm Im}\OO$, the space of imaginary octonions. There is a natural three-form on this space, defined by
$$\omega(x,y,z)={\rm Re}((xy)z), \qquad \forall x,y,z\in {\rm Im}\OO.$$
This three-form $\omega$ is invariant under $G_2={\rm Aut}(\OO)$, and it is known to be generic in the
sense that its ${\rm GL}(V_7)$-orbit is open in $\wedge^3V_7^*$.

\begin{proposition}
The hyperplane section $X_\omega$ of ${\rm Gr}(3,V_7)$ defined by the generic three-form
$\omega\in\wedge^3V_7^*$ is a compactification of $G_2/O_3$.
\end{proposition}

\begin{proof}
We denote by $e_0=1, e_1,\ldots , e_7$ the standard basis of $\OO$, whose multiplication table
is given by an oriented Fano plane, as in \cite{pr} (p. 105). A point of $X_\omega$ is the three-plane $L=\langle e_1,e_2,e_4\rangle$, and we claim that the stabilizer of $L$ in $G_2$ is isomorphic to
$O_3$. Indeed let $g_1,g_2,g_4$ be any orthonormal basis of $L$. Define $g_0=1$, $g_3=g_1g_2$,
$g_6=g_2g_4$, $g_7=g_4g_1$, and $g_5=(g_1g_2)g_4$.

\begin{lemma}
The endomorphism of $\OO$
sending $e_i$ to $g_i$ for $0\le i \le  7$, belongs to $G_2$.
\end{lemma}

This proves that an element
of $G_2$ that stabilizes $L$ is uniquely defined by its restriction to $L$, which can
be any element in the orthogonal group $O(L)\simeq O_3$.
\end{proof}

\subsection{$\BS_7$} \quad \\

\vspace{-3mm}
Let $V_{14}$ be a fourteen-dimensional vector space, endowed with a nondegenerate quadratic
form. We split $V_{14}=E\oplus F$ into two maximal isotropic spaces. These two spaces are
in duality with respect to the quadratic form. We choose a basis $e_1, \ldots ,e_7$ of $E$
and denote by $f_1,\ldots ,f_7$ the dual basis of $F$.

Recall that the two half-spin representations of ${\rm Spin}(V_{14})\simeq {\rm Spin}_{14}$
can be defined as
$S_+=\wedge^+E$ and $S_-=\wedge^-E$, the even and odd parts of the exterior algebra
of $E$. The action of the spin group is induced by the action of the Clifford algebra
of $V_{14}$ on $\wedge E$, where $E$ acts by exterior product and $F$ by contraction.
Here $S_+$ and $S_-$ are dual one to the other; alternatively, one can therefore
define $S_-$ as $\wedge^+F$.

The spinor variety $\BS_7$ is the closed orbit of the spin group inside $\BP S_+$.
It parametrizes the set of maximal isotropic subspaces of $V_{14}$ that meet $E$
in even dimension. This is done by associating to any such isotropic subspace $W$
a pure spinor $s_W$, uniquely defined (up to scalar) by the fact that its annihilator
(for the Clifford multiplication) is precisely $W$. For example $s_F=1$.

The action of the spin group on $\BP S_+$ and its dual $\BP S_-$ is known to be quasihomogeneous.
An explicit element in the open orbit of $\BP S_-$ is given \cite{sk}(pp. 131-132) by the class of the spinor
$$s^*=1+f_1f_2f_3f_7+f_4f_5f_6f_7+f_1f_2f_3f_4f_5f_6.$$
The stabilizer $H$ of $s^*$ in ${\rm Spin}(V_{14})$ is locally isomorphic to $G_2\times G_2$.
In fact, according to \cite{sk}, there are two distinguished seven dimensional subspaces
of $V_{14}$ that are preserved by $H$, namely
$$I_1=\langle e_7-f_7,e_1,e_2,e_3,f_1,f_2,f_3\rangle , \qquad
I_2=\langle e_7+f_7,e_4,e_5,e_6,f_4,f_5,f_6\rangle . $$
Note that the restriction of the quadratic form to either $I_1$ and $I_2$ is nondegenerate,
so that each of these spaces can be interpreted as a copy of the imaginary octonions. Moreover
the infinitesimal action of $\fg_2$ on $I_1$ is explicitely given by the matrices of the form
$$M=\begin{pmatrix}
0 & 2u & 2v & 2w & 2a & 2b & 2c \\
a & & & & 0 & w & -v \\
b & & A & & -w & 0 & u \\
c & & & & v & -u & 0 \\
u & 0 & -c & b & & & \\
v & c & 0 & -a & & -^tA & \\
w & -b & a & 0 & & &
\end{pmatrix},$$
where $A$ belongs to $\fsl_3$. Indeed, $\fg_2$ contains $\fsl_3$ as a subalgebra (generated by
the long root vectors), and  decomposes as an $\fsl_3$ module as $\fg_2=\fsl_3\oplus
V_3\oplus V_3^*$. Moreover the natural action of $\fg_2$ on $V_7=\CC\oplus V_3\oplus V_3^*$
is given by the previous matrices.

Similarly,
the infinitesimal action of $\fg_2$ on $I_2$ is given by the matrices
$$N=\begin{pmatrix}
0 & 2\lambda & 2\mu & 2\nu & 2d & 2e & 2f \\
d & & & & 0 & \nu & -\mu \\
e & & B & & -\nu & 0 & \lambda \\
f & & & & \mu & -\lambda & 0 \\
\lambda & 0 & -f & e & & & \\
\mu & f & 0 & -d & & -^tB & \\
\nu & -e & d & 0 & & &
\end{pmatrix}$$
where $B$ belongs to $\fsl_3$.

Now consider the space $W=\langle e_1-e_4,f_1+f_4,e_2-e_5,f_2+f_5,e_3-e_6,f_3+f_6,e_7\rangle$.
This is a maximal isotropic subspace of $V_{14}$, that intersects $E$ in four dimensions.
The associated pure spinor is
$$s_W=(e_1-e_4)(e_2-e_5)(e_3-e_6)e_7=e_1e_2e_3e_7-e_4e_5e_6e_7+\ldots .$$
One has $\langle s_W,s^*\rangle =0$: so $s_W$ is a pure spinor in the hyperplane section of
$\BS_7$ defined by $s^*$. Moreover, a straightforward computation shows that the
endomorphism of $V_{14}$ defined by a pair $(M,N)$ as above, preserves $W$ if and only
if $A=B$ and all the other coefficients vanish. In other words, the infinitesimal
stabilizer of $s_W$ in $\fg_2\times \fg_2$ is isomorphic to $\fsl_3$.

This implies that the $G_2\times G_2$-orbit of $s_W$ has dimension $20$, hence it must be open
in the hyperplane section of the spinor variety. We have proved:

\begin{proposition}
The general hyperplane section  of $\BS_7$
is a compactification of $G_2\times G_2/K$, where $K$ is locally isomorphic to
the diagonal $SL_3$.
\end{proposition}

\begin{remark}
Along the same line, we can get yet another proof of Proposition \ref{p.2S5}:  by \cite{sk} (p. 122), the general codimension 2 linear section $S_2$ of $\mathbb{S}_5$ is defined by    $s_1^*=1+f_1f_2f_3f_4$  and  $s_2^*=f_1f_5+f_2f_3f_4f_5$. Consider the space $W= \langle f_1, e_2-e_4, f_2+f_4, e_3-e_5, f_3+f_5 \rangle$ with associated pure spinor $s_W = (e_2-e_4)(e_3-e_5)$ which satisfies $\langle s_W,s_i^*\rangle =0, i=1, 2$, hence $s_W$ is on $S_2$.
By Proposition \ref{p.aut}, $\mathfrak{aut}(S_2) = \mathfrak{g}_2 \times \mathfrak{sl}_2$, which can be represented by the matrices in (5.40) \cite{sk} (p.122). It is now straight-forward to compute that the stabilizer of $s_W$ is defined by 8 linear equations, hence the orbit ${\rm Aut}^0(S_2) \cdot s_W$ is of dimension 8 and hence open in $S_2$, which proves that $S_2$ is quasi-homogeneous.  Now if we take $U= \langle e_1-e_3, f_1+f_3, e_2-e_4, f_2+f_4, f_5  \rangle$, then in the same way, we can show that $s_U \in S_2$ and  ${\rm Aut}^0(S_2) \cdot s_U$ is of dimension 6, which is the closed orbit in $S_2$.
\end{remark}

\subsection{$F_4/P_4$} \quad \\

\vspace{-3mm}
As was already noticed, $F_4/P_4$ is a generic hyperplane section of the Cayley plane
$E_6/P_1$, embedded in the projectivization of the minimal representation $V_{27}$ of $E_6$.
So we will consider a generic codimension two linear section of the Cayley plane.

Recall that $V_{27}$ can be identified with the exceptional Jordan algebra $H_3(\OO)$
of $3\times 3$ Hermitian matrices with octonionic coefficients:
$$M = \begin{pmatrix} r_1 & x_3 & x_2 \\ \overline{x}_3 & r_2 & x_1 \\ \overline{x}_2 & \overline{x}_1 &
r_3 \end{pmatrix}, \qquad x_1,x_2,x_3\in\OO, \; r_1,r_2,r_3\in\CC.$$
Moreover, one can understand the Cayley plane $E_6/P_1\subset \BP H_3(\OO)$ as the Zariski
closure of the set of matrices of the form
$$\begin{pmatrix} 1 & x & y \\ \overline{x} & \overline{x}x & \overline{x}y \\ \overline{y} & \overline{y}x &
\overline{y}y \end{pmatrix}, \qquad x,y\in\OO.$$
The full action of $E_6$ on the Cayley plane would be complicated to describe in full.
Let us just mention that the part of it that acts trivially on diagonal matrices respects
the nondiagonal blocks: it is made of transformations of type
$$\begin{pmatrix} r_1 & x_3 & x_2 \\ \overline{x}_3 & r_2 & x_1 \\ \overline{x}_2 & \overline{x}_1 &
r_3 \end{pmatrix}
\quad \mapsto \quad
\begin{pmatrix} r_1 & g_3(x_3) & g_2(x_2) \\ \overline{g_3(x_3)} & r_2 & g_1(x_1) \\ \overline{g_2(x_2)} & \overline{g_1(x_1)} & r_3 \end{pmatrix},$$
where $g_1,g_2,g_3$ belong to ${\rm GL}(\OO)$. For such a transformation to preserve the Cayley plane,
one needs the condition that
$$g_1(\overline{x}y)=\overline{g_3(x)}g_2(y)\qquad \forall x,y\in \OO.$$
By the celebrated triality principle, the set of such triples $(g_1,g_2,g_3)$ form a group
isomorphic to ${\rm Spin}_8$.

By \cite{sk}(p. 138), the action of $E_6$ on the Grassmannian of codimension two subspaces of $V_{27}$ is quasihomogeneous. Moreover, one defines a point in the open orbit by the
linear equations
$$r_1+r_2+r_3=r_1-r_3=0,$$
and the stabilizer of this point in $E_6$ is precisely the copy of ${\rm Spin}_8$ that we have
just described (up to a finite group). This ${\rm Spin}_8$ acts on the linear section $X_0$ of the
Cayley plane defined by our two linear equations. Consider the point $p$ of $X_0$ defined by the
matrix
$$\begin{pmatrix} 1 & i\sqrt{2}e_0 & e_0 \\ i\sqrt{2}e_0 & -2 & i\sqrt{2}e_0 \\ i\sqrt{2}e_0 &
i\sqrt{2}e_0 & 1 \end{pmatrix}.$$
The stabilizer of $p$ in ${\rm Spin}_8$ is the set of triples $(g_1,g_2,g_3)$  in ${\rm GL}(\OO)$
such that $g_i(e_0)=e_0, i=1,2,3$, and $g_1(\overline{x}y)=\overline{g_3(x)}g_2(y)$. Letting $x=e_0$
we get $g_1(y)=g_2(y)$, while letting $y=e_0$ we get $g_1(\overline{x})=\overline{g_3(x)}$. This
means that the triple $(g_1,g_2,g_3)$ is uniquely determined by $g_1$, which is submitted
to the condition that $g_1(\overline{x}y)=g_1(\overline{x})g_1(y)$. In other words, $g_1$ must
belong to ${\rm Aut}(\OO)=G_2$, and the stabilizer of $p$ in ${\rm Spin}_8$ is isomorphic to $G_2$.

Since ${\rm Spin}_8/G_2$ has the same dimension $14$ as the codimension two section $X_0$
of the Cayley plane, we conclude that:

\begin{proposition}
The generic hyperplane section of $F_4/P_4$, which is also a generic codimension two linear section
of the Cayley plane $E_6/P_1$, is a compactification of ${\rm Spin}_8/G_2$.
\end{proposition}


\bigskip
Chenyu Bai
(baichenyu14@mails.ucas.ac.cn )

School of Mathematical Sciences, University of Chinese Academy of Sciences, Beijing, China

\bigskip
Baohua Fu (bhfu@math.ac.cn)

MCM, AMSS, Chinese Academy of Sciences, 55 ZhongGuanCun East Road, Beijing, 100190, China
and School of Mathematical Sciences, University of Chinese Academy of Sciences, Beijing, China

 \bigskip
 Laurent Manivel (manivel@math.cnrs.fr)

 Institut de Math\'ematiques de Toulouse, UMR 5219, Universit\'e de Toulouse, CNRS, UPS IMT F-31062 Toulouse Cedex 9, France.


\begin{thebibliography}{KSWZ}
\bibitem[A]{A}
Akhiezer, Dmitry: Equivariant completions of homogeneous algebraic varieties by homogeneous divisors. Ann. Global Anal. Geom. 1 (1983), no. 1, 49–-78
\bibitem[AVE]{AVE}
Andreev, E. M.; Vinberg, E. B.; \`{E}la\v{s}vili, A. G., Orbits of
highest dimension of semisimple linear Lie groups. Funkcional.
Anal. i Priložen.  1  1967 no. 4, 3--7
\bibitem[BB]{BB}
Bien, Fr\'ed\'eric; Brion, Michel:
Automorphisms and local rigidity of regular varieties.
Compositio Math. 104 (1996), no. 1, 1–-26
\bibitem[B]{B}
Bott, Raoul: Homogeneous vector bundles. Ann. of Math. (2) 66 (1957), 203–-248
\bibitem[D]{D}
Demazure, Michel:  Automorphismes et d\'eformations des vari\'et\'es de Borel. Invent. Math. 39 (1977), no. 2, 179-–186
\bibitem[E]{E}
\`{E}la\v{s}vili, A. G.: Canonical form and stationary subalgebras
of points in general position for simple linear Lie groups.
Funkcional. Anal. i Prilo\v{z}en.  6  (1972), no. 1, 51--62
\bibitem[FH1]{FH1}
Fu, Baohua and Hwang, Jun-Muk: Classification of non-degenerate projective varieties with
non-zero prolongation and application to target rigidity,   Invent.
math. 189 (2012) 457-513
\bibitem[FH2]{FH2}
 Fu, Baohua and Hwang, Jun-Muk: Special birational transformations of type $(2,1)$. J. Algebraic Geom. 27 (2018), no. 1, 55--89
 \bibitem[FH3]{FH3}
 Fu, Baohua and Hwang, Jun-Muk:
Isotrivial VMRT-structures of complete intersection type,   Asian J. Math.  22 (2018), 331--352
\bibitem[HM]{HM} Hwang, Jun-Muk and Mok, Ngaiming: Prolongations of infinitesimal linear automorphisms of projective varieties and rigidity of rational homogeneous spaces of Picard number 1 under K\"ahler deformation. Invent. math. {\bf 160} (2005) 591--645
\bibitem[KW]{weymanE8}
Kra\'{s}kiewicz, Witold and  Weyman, Jerzy: Geometry of orbit closures for the representations associated to
gradings of Lie algebras of type $E_8$, preprint 2011.
\bibitem[K]{K}
Kuznetsov, Alexander:
On linear sections of the spinor tenfold, I,   arXiv:1801.00037
\bibitem[LM]{LM}
Landsberg, Joseph M. and Manivel, Laurent:
The projective geometry of Freudenthal's magic square.
J. Algebra 239 (2001), no. 2, 477-512
\bibitem[M]{pr} Manivel, Laurent:  Prehomogeneous spaces and projective geometry,
Rend. Sem. Mat. Univ. Politec. Torino Vol. 71, 1 (2013), 35-118
\bibitem[MS]{MS}
Merkulov, Sergei; Schwachh\"ofer, Lorenz: Classification of irreducible holonomies of torsion-free affine connections. Ann. of Math. (2) 150 (1999), no. 1, 77--149
\bibitem[Pa]{Pa}
Pasquier, Boris:
On some smooth projective two-orbit varieties with Picard number 1.
Math. Ann. 344 (2009), no. 4, 963-987
\bibitem[PP]{PP}
Pasquier, Boris; Perrin, Nicolas:  Local rigidity of quasi-regular varieties. Math. Z. 265 (2010), no. 3, 589–-600
\bibitem[PV]{PV}
Piontkowski, Jens; Van de Ven, Antonius:  The automorphism group of linear sections of the Grassmannians $G(1,N)$. Doc. Math. 4 (1999), 623-–664
\bibitem[P]{P}
Popov, Vladimir Leonidovich: Criteria for the stability of the action of a semisimple group on the factorial of a manifold.
Izv. Akad. Nauk SSSR Ser. Mat. 34 1970 523–-531
\bibitem[Ri]{Ri}
Richardson, R. W., Jr.: Principal orbit types for algebraic transformation spaces in characteristic zero.
Invent. Math. 16 (1972), 6–-14
\bibitem[Ru]{R}
Ruzzi, Alessandro: Geometrical description of smooth projective symmetric varieties with Picard number one. Transform. Groups 15 (2010), no. 1, 201-–226
\bibitem[SK]{sk} Sato M.,  Kimura T.,  A classification of irreducible prehomogeneous vector
spaces and their relative invariants, Nagoya  Math.  J. {\bf 65}  (1977),  1--155
\bibitem[V]{V}
Vinberg, \`{E}rnest Borisovich:  The Weyl group of a graded Lie algebra.  Izv. Akad. Nauk SSSR Ser. Mat. 40 (1976), no. 3, 488–-526, 709
\bibitem[W]{W}
Wahl, Jonathan: A cohomological characterization of $\mathbb{P}^n$.
Invent. Math. 72 (1983), no. 2, 315-–322
\end{thebibliography}
\end{document}